\documentclass[11pt,oneside]{amsart}

\usepackage{amssymb}
\usepackage{amsmath}
\usepackage{amsthm}
\usepackage[all]{xy}
\usepackage{url}

\usepackage{longtable}

\theoremstyle{plain}
\newtheorem{thm}{Theorem}[section]
\newtheorem{prop}[thm]{Proposition}
\newtheorem{lem}[thm]{Lemma}

\newtheorem{fact}[thm]{Fact}

\newtheorem{clm}[thm]{Claim}

\theoremstyle{definition}
\newtheorem{defn}[thm]{Definition}
\newtheorem{eg}[thm]{Example}

\theoremstyle{remark}
\newtheorem{rem}[thm]{Remark}

\DeclareMathOperator{\mult}{mult}
\DeclareMathOperator{\Pic}{Pic}

\DeclareMathOperator{\Nef}{Nef}

\DeclareMathOperator{\Spec}{Spec}
\DeclareMathOperator{\Proj}{Proj}
\DeclareMathOperator{\codim}{codim}
\DeclareMathOperator{\Supp}{Supp}

\DeclareMathOperator{\conv}{conv}

\DeclareMathOperator{\rank}{rank}

\DeclareMathOperator{\id}{id}
\DeclareMathOperator{\Zeros}{Zeros}

\def\N{\mathbb{N}}
\def\Z{\mathbb{Z}}
\def\Q{\mathbb{Q}}
\def\R{\mathbb{R}}
\def\C{\mathbb{C}}

\def\r+{\mathbb{R}_{\geq 0}}

\def\ep{\varepsilon}
\def\mr{M_{\R}}
\def\r+{{\R}_{\geq 0}}
\def\proj{\mathbb{P}}
\def\rarw{\rightarrow}

\def\*c{\C^{\times}}

\def\zp{\N \setminus 0}
\def\mb{\mathbf{m}}
\def\epp{\ep(X_P,L_P;1_P)}

\newcommand{\calo}{\mathcal {O}}

\newcommand{\m}{\mathfrak{m}}

\begin{document}

\title{Seshadri constants via toric degenerations}
\author{Atsushi Ito}
\address{Graduate School of Mathematical Sciences, 
the University of Tokyo, 3-8-1 Komaba, Meguro-ku, Tokyo 153-8914, Japan.}
\email{itoatsu@ms.u-tokyo.ac.jp}

\begin{abstract}
We give a method to estimate Seshadri constants
on toric varieties at any point.
By using the estimations and toric degenerations,
we can obtain some new computations or estimations of Seshadri constants
on non-toric varieties.
In particular,
we investigate Seshadri constants on hypersurfaces in projective spaces
and Fano 3-folds with Picard number one in detail.
\end{abstract}


\maketitle


\section{Introduction}\label{intro}

Let $X$ be a projective variety and $p \in X$ a point. 
In \cite{Dem},
Demailly defined an interesting invariant which measures the positivity of nef line bundles on $X$ at the point $p$.

\begin{defn}\label{def of sc_at a point}
Let $L$ be a nef line bundle on a projective variety $X$,
and take a (possibly singular) closed point $p \in X$. 
The \textit{Seshadri constant} of $L$ at $p$ is defined to be 
$$\ep(X,L;p)=\ep(L;p) := \max\{ \, t \geq 0 \, | \, \mu^* L-tE \ \text{is nef} \, \},$$
where $\mu: \widetilde{X} \rarw X$ is the blowing up at $p$ and $E= \mu^{-1}(p)$ is the exceptional divisor.
\end{defn}

\begin{rem}
Note that we can define $\ep(X,L;p)$ for a nef $\Q$-line bundle $L$ by the same definition.
\end{rem}

This invariant has some interesting properties.
For example, lower bounds for Seshadri constants imply jet separation of adjoint linear series \cite{Dem} and
lower bounds for Gromov width (an invariant in symplectic geometry) \cite{MP}.
Upper bounds sometimes give fibrations or foliations \cite{Na2}, \cite{HK}.
Seshadri constants are used to define the Ross-Thomas' slope stabilities for polarized varieties \cite{RT}.

\vspace{1mm}
Unfortunately it is very difficult to compute or estimate Seshadri constants in general.
To obtain upper bounds for Seshadri constants,
we have to show the non-nefness of a line bundle.
Thus it suffices to find a curve which has a negative intersection number with the line bundle.
But for lower bounds,
we have to show the non-existence of such curves,
and this is very difficult in general.
Many authors study surfaces,
but estimates for the Seshadri constant in higher dimensional cases are rare.
In higher dimensional cases,
the following results are known.

In \cite{EKL},
Ein, K\"uchle, and Lazarsfeld
show that $ \ep(X,L;p) \geq 1 / \dim X$
holds for a very general point $p \in X$
for any polarized variety $(X,L)$.
Seshadri constants on abelian varieties are studied by \cite{Na1}, \cite{La1}, \cite{Bau}, etc.
Even in toric cases,
Seshadri constants are computed only at a torus invariant point \cite{Di}, \cite{BDH+}.

\vspace{2mm}
The starting point of the present paper is to study Seshadri constants at any point on a toric variety.
Let $M$ be a free abelian group of rank $n$ and set $\mr =M \otimes_{\Z} \R$.
For an integral polytope $P \subset \mr$ of dimension $n$,
we can define a polarized toric variety $(X_P,L_P)$,
which has an action by the torus $T_M \cong (\*c)^n$.
As we will see later in Proposition \ref{sc at any point},
we can reduce the study of $\ep(X_P,L_P,p)$ for $p \in X_P$ to the case $p$ is torus invariant,
or $p$ is contained in the maximal orbit $O_P \subset X_P$.
Since the former case is already studied,
it suffices to consider the latter case.
By the torus action,
$\ep(X_P,L_P,p)$ is constant for any $p \in O_P$.
Thus we may assume $p=1_P \in O_P$,
where $1_P$ is the identity of $O_P \cong T_M$.

The following theorem extends Theorem 2.2 in \cite{Ec} for toric surfaces to all higher dimensional toric varieties.

\begin{thm}[cf.\ Theorem \ref{estimation by projection}]\label{intro thm}
Let $P$ be a rational polytope of dimension $n$ in $\mr$.
Fix a surjective group homomorphism $\pi_{\Z} : M \rarw \Z$,
and let $\pi : \mr \rarw \R$ be the linear projection induced from $\pi_{\Z}$.
For $t \in \pi(P) \cap \Q$ such that $P(t):= P \cap \pi^{-1}(t)$ is $n-1$-dimensional,
we have
$$\ep(X_P,L_P;1_P) \geq \min\{ \, |\pi(P)|, \ep(X_{P(t)},L_{P(t)};1_{P(t)}) \, \},$$
where $|\pi(P)|$ is the standard Euclidian length of $\pi(P) \subset \R$.
\end{thm}

Since $P(t)$ in the above theorem is an $n-1$-dimensional rational polytope,
we can obtain a lower bound for $\ep(X_P,L_P;1_P)$ inductively by this theorem.
The Eckl's proof relies on jet separation of linear systems.
On the other hand,
our proof is more geometric,
that is,
we use a rational map $X_P \dashrightarrow X_{P(t)}$ induced from the above data
to estimate intersection numbers with curves and line bundles.

We illustrate Theorem \ref{intro thm} by one example.

\begin{eg}\label{example1}
Let $P \subset \R^2$ be the convex hull of $(1,0),(0,1)$, and $(-1,-1)$ as in Figure \ref{figure1}.
Set $\pi : \R^2 \rarw \R$ to be the projection onto the second factor,
and take $t=0 \in \pi(P) \cap \Q$.

Since $\pi(P)=[-1,1]$ and $P(0)=P \cap \pi^{-1}(0) = [-1/2,1]$,
we have $|\pi(P)|=2$ and $\ep(X_{P(0)},L_{P(0)};1_{P(0)})=|P(0)|=3/2$.
Thus Theorem \ref{intro thm} says that
$$ \ep(X_P,L_P;1_P) \geq \min \{ 2,3/2 \}=3/2.$$
\end{eg}

\begin{figure}[htbp]
 \begin{center}
 \[
\begin{xy}
(10,0)="A",(0,10)="B",
(-10,-10)="C",
(9.9,0.1)="E",(-5,0.1)="F",
(9.8,-0.1)="I",(-5,-0.1)="J",
(-20,0)="1",(20,0)="2",
(0,-15)="3",(0,18)="4",
(50,-15)="7",(50,18)="8",
(11,3)*{1},(2,11)*{1},
(-12,-12)*{(-1,-1)},(-8,4)*{-\frac12},
(53,10)*{1},(53,0)*{0},(55,-10)*{-1},
(-15,18)*{\R^2},(55,18)*{\R}

\ar "1";"2"
\ar "3";"4"
\ar "7";"8"
\ar^{\pi} (27,0);(43,0)
\ar@{-} "A";"B"
\ar@{-} "B";"C"
\ar@{-} "C";"A"
\ar@{-} "E";"F"
\ar@{-} "I";"J"
\ar@{-} (9.9,.2);(-4.9,.2)
\ar@{-} (9.8,-.2);(-5,-.2)
\ar@{-} (49,10);(51,10)
\ar@{-} (49,0);(51,0)
\ar@{-} (49,-10);(51,-10)
\ar@{-} (49.9,10);(49.9,-10)
\ar@{-} (50.1,10);(50.1,-10)
\ar@{-} (49.8,10);(49.8,-10)
\ar@{-} (50.2,10);(50.2,-10)
\end{xy}
\]
 \end{center}
 \caption{}
 \label{figure1}
\end{figure}

\vspace{4mm}

Lower bounds in toric case sometimes give interesting lower bounds in non-toric case.
To state this,
we introduce the following notation.

\begin{defn}\label{def of sc_at a general pt}
Let $L$ be a nef line bundle on a projective variety $X$.
The Seshadri constant $\ep(X,L;1)$ of $L$ at a very general point is defined to be
$$ \ep(X,L;1) := \ep(X,L;p) $$
for a very general point $p \in X$.
Note that $\ep(X,L;1)$ is well-defined,
i.e.,
$\ep(X,L;p) $ is constant for a very general $p$ (see \cite[Example 5.1.11]{La2}).
\end{defn}

\vspace{2mm}
Since ampleness is an open condition in a flat family,
Seshadri constants are essentially lower semicontinuous.
Thus if we can degenerate a given polarized variety $(X,L)$ to a toric one,
we obtain a lower bound for $\ep(X,L;1)$ by Theorem \ref{intro thm}.

\vspace{1mm}
We illustrate this strategy with cubic surfaces.
Let $S \subset \proj^3$ be a very general cubic surface.
The surface $S$ degenerates to the hypersurface
$$S_0:=\Zeros(T_0^3 - T_1 T_2 T_3) \subset \proj^3,$$
where $T_0,\ldots,T_3$ are the homogeneous coordinates on $\proj^3$.
It is easy to see that $(S_0,\calo_{S_0}(1))$ is isomorphic to the polarized toric variety $(X_P,L_P)$
for $P \subset \R^2$ in Example \ref{example1}.
From the lower semicontinuity and Example \ref{example1},
we have $$\ep(S,\calo_S(1);1) \geq \ep(S_0,\calo_{S_0}(1);1) = \ep(X_P,L_P;1) \geq 3/2.$$
In fact,
it is well known
that $\ep(S,\calo_S(1);1) = 3/2$ holds (cf.\ \cite[Example 2.1]{ST}).

\vspace{1mm}
By similar arguments,
we can prove the following theorem.

\begin{thm}[cf.\ Proposition \ref{a point in c.i.}, Theorem \ref{sc of hypersurfaces}]\label{intro thm 3}
For a very general hypersurface $X \subset \proj^{n+1}$ of degree $d$,
we have
$$\lfloor \sqrt[n]{d} \rfloor \leq \ep(X,\calo_X(1);1) \leq \sqrt[n]{d} .$$
\end{thm}

\vspace{1mm}
\begin{rem}
In Example \ref{examples of hypersurf at a point},
we will show a refined statement of Theorem \ref{intro thm 3},
which recovers the above lower bound $3/2$ for cubic surfaces.
\end{rem}

\vspace{2mm}
Iskovskih classified smooth Fano 3-folds with Picard number $1$ into $17$ families in \cite{Is1}, \cite{Is2}.
By \cite{Gal}, \cite{CI}, and \cite{ILP}, 
toric degenerations of such 3-folds 
are well studied.
Using the degenerations,
we obtain the following explicit values.

\begin{thm}[=Theorem \ref{Fano 3-fold}]\label{intro thm 4}
For each family of smooth Fano 3-folds with Picard number $1$,
$\ep(X,-K_X;1)$ is as in Table \ref{table},
where $X$ is a very general member in the family.
We refer the reader to \cite{IsP} for the description of $X_{12}, X_{16},X_{18}$, and $X_{22}$.
\end{thm}

\begin{center}
\begin{longtable}{||c|c|c|p{7cm}|c||}
\hline

  No. & Index & $(-K_X)^3$ &
\begin{minipage}[c]{7cm}
\centering
\vspace{.2cm}

  Description

\vspace{.2cm}

\end{minipage}
  &
$\ep(X,-K_X;1)$
\\
  \hline
  \hline

  1 & 1 & 2 &
\begin{minipage}[c]{7cm}
\vspace{.1cm}

a hypersurface of degree $6$ in $\proj(1,1,1,1,3)$

\vspace{.1cm}

\end{minipage}
&
6/5
    \\
  \hline
  2 & 1 & 4 &
\begin{minipage}[c]{7cm}
\vspace{.1cm}

  the general element of the family is a quartic

\vspace{.1cm}

\end{minipage}
&
4/3
  \\
  \hline
  3 & 1 & 6 &
\begin{minipage}[c]{7cm}
\vspace{.1cm}

a complete intersection of a quadric and a cubic

\vspace{.1cm}

\end{minipage}
&
3/2
\\
  \hline
  4 & 1 & 8 &
\begin{minipage}[c]{7cm}
\vspace{.1cm}

a complete intersection of three quadrics

\vspace{.1cm}

\end{minipage}
&
2
\\
  \hline
  5 & 1 & 10 &
\begin{minipage}[c]{7cm}
\vspace{.1cm}

  the general element is
  a section of $G(2,5)$ in Pl\"{u}cker embedding by 2 hyperplanes and a quadric

\vspace{.1cm}

\end{minipage}
&
2
\\
  \hline
  6 & 1 & 12 &
\begin{minipage}[c]{7cm}
\vspace{.1cm}

  $X_{12} \subset \proj^8$

\vspace{.1cm}

\end{minipage}
&
2
\\
  \hline
  7 & 1 & 14 &
\begin{minipage}[c]{7cm}
\vspace{.1cm}

 a section of $G(2,6)$ by 5 hyperplanes in Pl\"{u}cker embedding

\vspace{.1cm}

\end{minipage}
&
2
\\
  \hline
  8 & 1 & 16 &
\begin{minipage}[c]{7cm}
\vspace{.1cm}

  $X_{16} \subset \proj^{10}$

  \vspace{.1cm}

\end{minipage}
&
2
\\
  \hline
  9 & 1 & 18 &
\begin{minipage}[c]{7cm}
\vspace{.1cm}

  $X_{18} \subset \proj^{11}$

\vspace{.1cm}

\end{minipage}
&
2
\\
  \hline
  10 & 1 & 22 &
\begin{minipage}[c]{7cm}
\vspace{.1cm}

  $X_{22} \subset \proj^{13}$

\vspace{.1cm}

\end{minipage}
&
2
\\
  \hline
  11 & 2 & $8\cdot 1$ &
\begin{minipage}[c]{7cm}
\vspace{.1cm}

 a hypersurface of degree $6$ in $\proj(1,1,1,2,3)$
 
 \vspace{.1cm}

\end{minipage}
&
2
\\
  \hline
  12 & 2 & $8\cdot 2$ &
  \begin{minipage}[c]{7cm}
\vspace{.1cm}

 a hypersurface of degree $4$ in $\proj(1,1,1,1,2)$

\vspace{.1cm}

\end{minipage}
&
2
\\
  \hline
  13 & 2 & $8\cdot 3$ &
  \begin{minipage}[c]{7cm}
\vspace{.1cm}

a cubic

\vspace{.1cm}

\end{minipage}
&
2
\\
  \hline
  14 & 2 & $8\cdot 4$ &
\begin{minipage}[c]{7cm}
\vspace{.1cm}

 a complete intersection of two quadrics

\vspace{.1cm}

\end{minipage}
&
2
\\
  \hline
  15 & 2 & $8\cdot 5$ &
\begin{minipage}[c]{7cm}
\vspace{.1cm}

 a section of $G(2,5)$ by 3 hyperplanes in Pl\"{u}cker embedding

\vspace{.1cm}

\end{minipage}
&
2
\\
  \hline
  16 & 3 & $27\cdot 2$ &
\begin{minipage}[c]{7cm}
\vspace{.1cm}

 a quadric

\vspace{.1cm}

\end{minipage}
&
3
\\
  \hline
  17 & 4 & $64 \cdot 1$ &
\begin{minipage}[c]{7cm}
\vspace{.1cm}

  $\proj^3$

\vspace{.1cm}

\end{minipage}
&
4
\\

 \hline

\caption{\label{table}Seshadri constants on Fano 3-folds with $\rho=1$}

\end{longtable}
\end{center}

The definition of Seshadri constants can be easily generalized to multi-point cases 
(cf.\ \cite[Definition 5.4.1]{La2}, \cite[Definition 1.9]{BDH+}).
We define the multi-point version at very general points.

\begin{defn}\label{def of sc_multi}
Let $L$ be a nef line bundle on a projective variety $X$.
For a positive integer $r $
and $\mb=(m_1,\ldots,m_r) \in \r+^r \setminus 0$,
the Seshadri constant $\ep(X,L; \mb)$ of $L$ at very general points with weight $\mb$ is defined to be
$$\ep(X,L; \mb) := \max\{ \, t \geq 0 \, | \, \mu^* L-t \, \sum_{i=1}^r m_i E_i \ \text{is nef} \, \} ,$$
where $\mu: \widetilde{X} \rarw X$ is the blowing up at very general $r$ points $p_1,\ldots,p_r \in X$ and $E_i= \mu^{-1}(p_i)$
is the exceptional divisor over $p_i$.
\end{defn}

Theorem \ref{intro thm 3} can be generalized to the multi-point case as follows.

\begin{thm}[=Theorem \ref{sc of hypersurfaces}]\label{intro thm 5}
For a very general hypersurface $X \subset \proj^{n+1} $ of degree $d$,
we have
$$\lfloor \sqrt[n]{d / (m_1^n+\cdots + m_r^n)} \rfloor \leq \ep(X,\calo_X(1);\mathbf{m}) \leq \sqrt[n]{d / (m_1^n+\cdots + m_r^n)} $$
for any $\mathbf{m}=(m_1,\ldots,m_r)  \in \N^r \setminus 0.$
\end{thm}

\vspace{1mm}
This paper is organized as follows. In Section 2, we recall some notations, conventions,
and well known facts about Seshadri constants.
In Section 3, we examine Seshadri constants on toric varieties
and establish Theorem \ref{intro thm}.
In Section 4, we prove Theorems \ref{intro thm 3}, \ref{intro thm 4}, and \ref{intro thm 5}.
Throughout this paper,
we consider varieties or schemes over the complex number field $\C$.

\subsection*{Acknowledgments}
The author would like to express his gratitude to his advisor Professor Yujiro Kawamata 
for his valuable advice, comments, and warm encouragement.
He is grateful to Professor Robert Lazarsfeld
for reading the draft and giving me helpful advice and comments.
He wishes to thank Professors Yoshinori Gongyo, Yasunari Nagai, Shinnosuke Okawa, Hiromichi Takagi, and Kiwamu Watanabe
for useful comments and suggestions.
He would also like to thank Makoto Miura and Taro Sano
for helpful comments.

The author was supported by the Grant-in-Aid for Scientific Research
(KAKENHI No. 23-56182) and the Grant-in-Aid for JSPS fellows.

\section{Preliminary}\label{notation}

\subsection{Notations and conventions}
We denote by $\N,\Z,\Q,\R$, and $\C$ the set of all 
natural numbers, integers, rational numbers, real numbers, and complex numbers respectively.
In this paper, $\N$ contains $0$.
We define $\r+:=\{ x \in \R \,  | \, x \geq 0\}$.
For $x \in \R$,
$\lfloor x \rfloor,\lceil x \rceil \in \Z$ are the round down and the round up of $x$ respectively.
We denote by $e_1,\ldots,e_n$ the standard basis of $\Z^n$ or $\R^n$.

Unless otherwise stated,
$M$ stands for a free abelian group of rank $n \in \N$ throughout this paper.
We define $M_K:=M \otimes_{\Z} K$ for any field $K$.
For a subset $S \subset \mr$, we denote the convex hull of $S$ by $\conv(S)$.
We write $\Sigma (S)$ for the closed convex cone $S$ spans.
For $t \in \R$,
$tS:=\{tu \, | \, u \in S \}$.
For $u \in \mr$,
$S+u:=\{u' + u \, | \, u' \in S \}$ is
the parallel translation of $S$ by $u$.

A subset $P \subset \mr$ is called a \textit{polytope} if it is the convex hull of a finite set in $\mr$.
A polytope $P$ is \textit{integral} (resp.\ \textit{rational}) if all vertices are in $M$ (resp.\ $M_{\Q}$).
The \textit{dimension} of $P$ is the dimension of the affine space spanned by $P$.
When $\sigma$ is a face of a polytope $P \subset \mr$,
we write $\sigma \prec P$.
For a polytope $P \subset \mr$ with $\rank M=1$,
we denote by $|P|_M$ or $|P|$ the standard Euclidean length of $P$
under an identification of $M \subset \mr$ with $\Z \subset \R$.

For free abelian groups $M$ and $M'$ of rank $n$ and $r$,
a linear map $\pi:\mr \rarw M'_{\R}$
is called a \textit{lattice projection}
if $\pi$ is induced from a surjective group homomorphism $M \rarw M'$.

For a variety $X$, we say a property holds at a \textit{very general} point of $X$
if it holds for all points in the complement of the union of countably many proper subvarieties.

Throughout this paper, a divisor means a Cartier divisor.
We use the words ``divisor'',``line bundle'', and ``invertible sheaf'' interchangeably.

We call a pair $(X,L)$ a ($\Q$-)\textit{polarized variety} 
if $X$ is a projective variety and $L$ is an ample ($\Q$-)line bundle on $X$.

\subsection{Seshadri constants}
We refer the reader to \cite[Chapter 5]{La2} for the basic treatment of Seshadri constants.
The following facts are well known and follow easily from the definition of Seshadri constants.
We will use these facts later repeatedly.
\begin{fact}\label{well_known_facts}
Let $L$ be a nef line bundle on a projective variety $X$.
Then the following hold.
\begin{itemize}
\item[(1)] We have an inequality $\ep(L;1) \leq \sqrt[n]{L^n}$, 
where $n$ is the dimension of $X$.
Similarly,
$\ep(L;\mathbf{m}) \leq \sqrt[n]{L^n/ (m_1^n+\cdots+m_r^n)}$ holds for any $\mb=(m_1,\ldots,m_r) \in \r+^r \setminus 0$.
\item[(2)] For a subvariety $Y$ of $X$,
$\ep(X,L;p) \leq \ep(Y,L |_Y;p) $ holds for any $p \in Y \subset X$.
\item[(3)] We have
\begin{center}
$\displaystyle \ep(X,L;p)=\inf_C \left\{\dfrac{C.L}{\mult_p(C)}\right \}$,
\end{center}
where $C$ moves all reduced and irreducible curves on $X$ passing through $p$,
and $\mult_p(C)$ is the multiplicity of $C$ at $p$.
\end{itemize}
\end{fact}

\section{Seshadri constants on toric varieties}

In this section, we investigate Seshadri constants on toric varieties
and prove Theorem \ref{intro thm}.
We treat non-normal toric varieties in this paper.
We refer the reader to \cite{Fu} for toric varieties.

\begin{defn}\label{def of toric}
Let $\Gamma \subset \N \times M$ be a finitely generated subsemigroup
such that $\Gamma \cap (\{0\} \times M) = \{0\}$ 
and $\Gamma$ generates $\Z \times M$ as a group.
We define a $\Q$-\textit{polarized toric variety} $(X(\Gamma),L(\Gamma))$ as follows.
$$ (X(\Gamma),L(\Gamma)):=(\Proj \C[\Gamma],\calo _{\Proj \C[\Gamma]}(1)).$$
The torus $T_M:= \Spec \C[M]$ naturally acts on $(X(\Gamma),L(\Gamma))$.

The \textit{moment polytope} $\Delta(\Gamma)$ of $(X(\Gamma),L(\Gamma))$ is defined to be 
$$\Delta(\Gamma):=\Sigma (\Gamma) \cap (\{1\} \times \mr) \subset \{1\} \times \mr,$$
which can be regarded as a rational polytope in $\mr$ naturally.

For a rational polytope $P \subset \mr$ of dimension $n$,
we define the \textit{normal $\Q$-polarized toric variety} $(X_P, L_P)$ by
$$(X_P,L_P):=(X(\Gamma_{P}),L(\Gamma_{P})),$$
where $\Gamma_{P}:=\Sigma (\{1\} \times P) \cap (\N \times M).$
We write the maximal orbit of $X_P$ as $O_P$,
and denote by $1_P \in O_P=T_M$
the identity of the torus.
For a face $\sigma$ of $P$,
there is a natural closed embedding $X_{\sigma} \hookrightarrow X_P$.
Hence we can regard $X_{\sigma}$ as a closed subvariety of $X_P$,
and $O_{\sigma}$ is considered as a $T_M$-orbit in $X_P$.

\end{defn}

\begin{rem}\label{normalization}
(1) For any semigroup $\Gamma$ as in Definition \ref{def of toric},
the normalization of $(X(\Gamma),L(\Gamma))$ is
$\mu : (X_{\Delta(\Gamma)},L_{\Delta(\Gamma)}) \rightarrow (X(\Gamma),L(\Gamma))$ induced by 
$\Gamma \hookrightarrow \Sigma (\Gamma) \cap (\N \times M)$ (see \cite[Exercise 4.22]{Ei}).
When $P$ is an integral polytope,
$L_P$ is a line bundle.\\
(2) For any integral polytope $P \subset \mr$ of dimension $n$
and any face $\sigma \prec P$,
$\ep(X_P,L_P;p)$ is constant for $p \in O_{\sigma}$ because of the torus action.
In particular,
$\ep(X_P,L_P;1)$ (in the sense of Definition \ref{def of sc_at a general pt})
coincides with $\ep(X_P,L_P;1_P)$.
\end{rem}

\subsection{At a point in the maximal orbit}

In this subsection,
we estimate $\epp$ for an integral polytope $P$.

\begin{lem}\label{orbit}
Let $\pi:\mr \rarw M'_{\R}$ be a lattice projection
with $\rank M=n, \rank M'=r$,
and $P \subset \mr$ an integral polytope of dimension $n$.
Then the closure $(\overline{T_{M'}}, L_P|_{\overline{T_{M'}}})$ of $T_{M'}$ in $X_P$
is a polarized toric variety
whose moment polytope is $\pi(P) \subset M'_{\R}$,
where $T_{M'} \hookrightarrow T_{M}=O_P \subset X_P$ is induced by the surjection $\pi | _{M} :M \rarw M'$.
 \end{lem}

\begin{proof}
Consider the following commutative diagram
\[\xymatrix{
\Gamma_{P}=\Sigma (\{1\} \times P)\cap (\N \times M) \ar[d] \ar@{^{(}->}[r] \ar@{}[dr]|\circlearrowleft & \N \times M \ar@{->>}[d]^{\id_{\N} \times \pi |_M} \\
\Gamma_{\pi(P)}=\Sigma (\{1\} \times \pi(P)) \cap (\N \times M') \ar@{^{(}->}[r] & \N \times M' , \\
}\]
and set $\Gamma '= (\id_{\N} \times \pi |_{M})(\Gamma_{P})$.
The semigroup $\Gamma '$ generates $\Z \times M'$ as a group
and $\Delta(\Gamma')=\pi(P)$
because $\dim P=n$ and $\pi |_{M}:M \rarw M'$ is surjective.

The above diagram induces 
\[\xymatrix{
X_P=\Proj\C[\Gamma_{P} ]  \ar@{}[dr]|\circlearrowleft & T_M \ar@{_{(}->}[l]^(0.3){i} \\
X(\Gamma')=\Proj\C[\Gamma '] \ar@{^{(}->}[u]^{\iota } & T_{M'} \ar@{^{(}->}[u] \ar@{_{(}->}[l]^(0.3){i'}, \\
}\]
where $\iota $ is a closed embedding, and $i,i'$ are open immersions.
Therefore, we have $(\overline{T_{M'}},L_P|_{\overline{T_{M'}}})=(X(\Gamma '),L(\Gamma '))$,
whose moment polytope is $\Delta(\Gamma')=\pi(P)$.
\end{proof}

\vspace{1mm}

Let $\pi:\mr \rarw M'_{\R}$ and $P$ be as in Lemma \ref{orbit},
and set
$P(u') = \pi^{-1}(u') \cap P $ for $u' \in \pi(P) \cap M'_{\Q} .$

\begin{figure}[htbp]
 \begin{center}
\[
\begin{xy}
(20,3)="A",(35,13)="B",
(15,30)="C",(5,23)="D",
(5,10)="E",(5,20)="F",
(26.7,20)="G",
(-5,0)="1",(40,0)="2",
(0,-5)="3",(0,35)="4",
(70,-5)="7",(70,35)="8",
(74,20.5)*{u'},(18,15)*{P},
(-5,33)*{\mr},(76,33)*{\mr'},
(35.5,28)*{P(u')}

\ar "1";"2"
\ar "3";"4"
\ar "7";"8"
\ar^{\pi} (45,10);(60,10)
\ar@{-} "A";"B"
\ar@{-} "B";"C"
\ar@{-} "C";"D"
\ar@{-} "D";"E"
\ar@{-} "E";"A"
\ar@{-} "F";"G"
\ar@{-} (69,20);(71,20)
\ar@/_/ (30,28.5);(18,20.5)
\end{xy}
\]
 \end{center}
 \caption{}
 \label{figure2}
\end{figure}

An splitting $M \cong \ker \pi |_M \oplus M'$ of $0 \rarw \ker \pi |_M \rarw M \stackrel{\pi |_M}{\rarw} M' \rarw 0$
induces an identification of $\pi^{-1}(u')$ with $\ker \pi$,
hence we can consider $P(u')$ as a rational polytope in $\ker \pi =({\ker \pi |_M})_{\R}$.
Assume that the dimension of $P(u')$ is $n-r$.
Then $P(u')$ defines the $\Q$-polarized toric variety $(X_{P(u')},L_{P(u')})$,
and the isomorphic class of $(X_{P(u')},L_{P(u')})$
does not depend on the choice of the splitting $M \cong \ker \pi |_M \oplus M'$.

\begin{lem}\label{construction of fibration}
Let $\pi:\mr \rarw M'_{\R}$ and $P$ be as in Lemma \ref{orbit},
and take $u' \in \pi(P) \cap M'_{\Q} $
such that $\dim P(u')= n-r$. 
Then,
there exists a generically surjective rational map $\varphi  : X_P \dashrightarrow X_{P(u')}$
such that for any resolution $\mu : Y \rightarrow X_P $ of the indeterminacy of $\varphi $,
the following conditions hold.
\begin{itemize}
\item[(i)] $\mu^* L_P - f^* L_{P(u')}$ is $\Q$-effective, where $f=\varphi \circ \mu $,
\item[(ii)]  $\mu(f^{-1}(1_{P(u')})) \cap O_P=T_{M'}$ holds for $1_{P(u')} \in O_{P(u')} \subset X_{P(u')}$.

\[\xymatrix{
Y \ar[r]^{\mu} \ar[dr]_f & X_P \ar@{-->}[d]^{\varphi } \\
 & X_{P(u')} \\
}\]

\end{itemize}

\end{lem}

\begin{proof}
By considering $kP$ for sufficiently large and divisible $k \in \N$,
we may assume that $u'$ is contained in $M'$ and $P(u')$ is an integral polytope.
Furthermore, by considering $P-u$ for $u \in (\pi |_M)^{-1}(u')$,
we may assume $u'=0 \in M'$.
Hence $P(u')$ is an integral polytope in $\pi^{-1}(0)=(\ker \pi |_M)_{\R}$.

Consider the following commutative diagram
\[\xymatrix{
\Gamma_P =\Sigma (\{1\} \times P) \cap (\N \times M)  \ar@{^{(}->}[r] \ar@{}[dr]|\circlearrowleft & \N \times M  \\
\Gamma_{P(u')}=\Sigma (\{1\} \times P(u'))\cap (\N \times \ker \pi |_M) \ar@{^{(}->}[u]^{\psi} \ar@{^{(}->}[r]
& \N \times \ker \pi |_M. \ar@{^{(}->}[u] \\
}\]
This commutative diagram induces 
\[\xymatrix{
X_P=\Proj\C[\Gamma_P ] \ar@{-->}[d]^{\varphi} \ar@{}[dr]|\circlearrowleft & O_P = T_M \ar@{_{(}->}[l] \ar@{->>}[d]^{\varphi |_{T_M}} \\
X_{P(u')}=\Proj\C[\Gamma_{P(u')}] & O_{P(u')}=T_{\ker{\pi |_M}}. \ar@{_{(}->}[l] \\
}\]
Then $\varphi$ is generically surjective because $\varphi|_{T_M}$ is surjective.
We show this $\varphi $ satisfies (i) and (ii) in the statement of this lemma.

Since $\varphi |_{T_M}^{-1}(1_{P(u')})=T_{M'}$,
(ii) is clear.

Let $\mu : Y \rightarrow X_P $ be a resolution of the indeterminacy of $\varphi $.
Both $X_P,X_{P(u')}$ are normal and
$\mu,f $ have connected fibers.
Thus we have
$$\bigoplus_{k \in \N} H^0(Y,k f^*L_{P(u')})=\bigoplus_{k \in \N} H^0(X_{P(u')},k L_{P(u')})=\Gamma_{P(u')},$$
$$\bigoplus_{k \in \N} H^0(Y,k \mu^*L_P) =\bigoplus_{k \in \N} H^0(X_P,k L_P) =\Gamma_P.$$
Therefore an injection $f^*L_{P(u')} \hookrightarrow \mu^*L_P$ is induced from the injection
$\psi: \Gamma_{P(u')} \rarw \Gamma_P$.
Hence (i) holds.
\end{proof}

We need one more lemma,
which states that lower and upper bounds for Seshadri constants are obtained from surjective morphisms.
Recall that for a line bundle $L$ on a projective variety,
the \textit{stable base locus} of $L$ is defined to be $\mathbf{B}(L) := Bs(|kL|)$ for sufficiently large and divisible $k \in \N$,
which does not depend on such $k$; see \cite[Remark 2.1.24]{La2}.

\begin{lem}\label{estimate by fibration}
Let $f:Y \rightarrow Z$ be a surjective morphism between projective varieties.
Assume that $L,L'$ are nef $\Q$-divisors on $Y,Z$ respectively such that $L-f^*L'$ is $\Q$-effective.
Then
$$\min \{ \min_{1 \leq i \leq r} \ep(Y_i,L|_{Y_i};y), \ep(Z,L';f(y)) \} \leq \ep(Y,L;y)
\leq \min_{1 \leq i \leq r} \ep(Y_i,L|_{Y_i};y)$$
holds for $y \not\in \mathbf{B}(L-f^*L')$,
where $Y_1,\ldots,Y_r$ are all the irreducible components of $f^{-1}(f(y))$ containing $y$ with the reduced structures.
\end{lem}

\begin{proof}

We may assume $L$ and $L'$ are ample.
In fact,
for nef $L,L'$,
choose ample divisors $A,A'$ on $Y,Z$ such that $y \not \in \mathbf{B}(A - f^*A')$,
and consider $L+\delta  A,L'+ \delta  A'$ with $ \delta  >0$ instead of $L,L'$.
Then we can show this lemma from ample cases by $\delta \rightarrow 0$.

The second inequality is clear by the definition of Seshadri constants,
thus it is enough to show the first one.
For the sake of simplicity, we set $z=f(y)$.
 
Fix a curve $C \subset Y$ containing $y$.
By Fact \ref{well_known_facts} (3),
it suffices to show
\begin{align}
\tag{$*$}
\min \{ \min_i \ep(Y_i,L|_{Y_i};y), \ep(Z,L';z) \} \leq \dfrac{C.L}{\mult_y(C)}.
\end{align}

\begin{itemize}
\item[Case 1.] $C \subset f^{-1}(z)$.\\
Since $C \subset Y_i$ for some $i$,
$\dfrac{C.L}{\mult_y(C)} = \dfrac{C.L|_{Y_i}}{\mult_y(C)} \geq \ep(Y_i,L|_{Y_i};y)$ holds.\\
\item[Case 2.] $C \not \subset f^{-1}(z)$.\\
Set $C'=f(C)$ with the reduced structure,
and fix a rational number $0 < t < \varepsilon(Z,L';z)$.
Then for any sufficiently large and divisible $k \in \N$,
there exists $D' \in |kL' \otimes \mathfrak{m}_{z}^{kt}|$ such that $C' \not\subset \Supp D'$
by the ampleness of $L'$ and \cite[Lemma 5.4.24]{La2}.
Since $f^*D' \in |kf^*L' \otimes \mathfrak{m}_{y}^{kt}|$ and $C \not\subset \Supp f^*D'$,
we have
\begin{eqnarray*}
k(C.L) &=& kC.(L-f^*L' + f^*L') \\
       &=& kC.(L-f^*L') + C.f^*D' \\
      &\geq& C.f^*D' \\
      &\geq& kt \cdot \mult_y(C).
\end{eqnarray*}
Note $C.(L-f^*L') \geq 0 $ holds by the assumption $y \not\in \mathbf{B}(L-f^*L') $.
Therefore $\dfrac{C.L}{\mult_y(C)} \geq t$ holds
and we have $\dfrac{C.L}{\mult_y(C)} \geq \ep(Z,L';z)$
by $t \rightarrow \ep(Z,L';z)$.
\end{itemize}

Thus for any curve $C \subset Y$ containing $y$,
$(*)$ holds.
\end{proof}

By combining Lemmas \ref{orbit}, \ref{construction of fibration}, and \ref{estimate by fibration},
we obtain the following theorem,
which will be used to estimate $\epp$.
Theorem \ref{intro thm} is nothing but the last statement in this theorem.
The lower bound in this theorem is essentially a generalization of an Eckl's result \cite[Theorem 2.2]{Ec},
which is the case $n=2,r=1$.
In contrast with our geometric approach, Eckl's proof is more algebraic.

\begin{thm}\label{estimation by projection}
Let $\pi : \mr \rarw \mr' $ be a lattice projection for free abelian groups $M$ and $M'$ of rank $n$ and $r$.
Let $P \subset \mr$ be an $n$-dimensional integral polytope,
and fix $u' \in \pi(P) \cap M'_{\Q}$
such that $\dim P(u')=n-r$.
We have
\begin{eqnarray*}
\lefteqn{\min \{ \ep(X_{\pi(P)},L_{\pi(P)};1_{\pi(P)}), \ep(X_{P(u')},L_{P(u')};1_{P(u')}) \}  }\hspace{4cm} \\ 
 &\leq & \ep(X_P,L_P;1_P) \ \leq \ \ep(X_{\pi(P)},L_{\pi(P)};1_{\pi(P)}).
\end{eqnarray*}

In particular, we obtain
$$\min \{ \, |\pi(P)|, \ep(X_{P(u')},L_{P(u')};1_{P(u')}) \} 
 \leq  \ep(X_P,L_P;1_P) \ \leq \ |\pi(P)| $$
when $r=1$.
\end{thm}

\begin{proof}
Let $\varphi  : X_P \dashrightarrow X_{P(u')}$ be the rational map defined in Lemma \ref{construction of fibration}.
For a toric resolution $\mu:Y \rightarrow X_P$ of the indeterminacy of $\varphi $,
the stable base locus $\mathbf{B}(\mu^*L_P - f^*L_{P(u')})$ is contained in $Y\setminus O$,
where $O$ is the maximal orbit of $Y$.
By applying Lemma \ref{estimate by fibration} to $f:Y \rightarrow X_{P(u')},\mu^*L_P,L_{P(u')}$,
and the identity $1_Y$ of the torus $ O \subset Y$,
we have
\begin{eqnarray*}
\lefteqn{\min \{ \ep(Y_1,(\mu^*L_P)|_{Y_1};1_Y), \ep(X_{P(u')},L_{P(u')};1_{P(u')} ) \} }\hspace{3cm} \\ 
 &\leq & \ep(Y,\mu^*L_P;1_Y) \ \leq \ \ep(Y_1,(\mu^*L_P)|_{Y_1};1_Y),
\end{eqnarray*}
where $Y_1$ is the irreducible component of $f^{-1}(1_{P(u')})$ containing $1_Y$.
Since $\mu:Y \rightarrow X_P$ and $f|_{Y_1}:Y_1 \rightarrow \overline{T_{M'}} (\subset X_P)$
are birational and isomorphic around $1_Y$ from the proof of Lemma \ref{construction of fibration},
it holds that
$\ep(Y,\mu^*L_P;1_Y)=\ep(X_P,L_P;1_P)$,
$\ep(Y_1,(\mu^*L_P)|_{Y_1};1_Y)=\ep(\overline{T_{M'}}, L_P|_{\overline{T_{M'}}};1_P)$.
The normalization of $(\overline{T_{M'}}, L_P|_{\overline{T_{M'}}})$ is $(X_{\pi(P)},L_{\pi(P)})$ by Lemma \ref{orbit}.
Since Seshadri constants at very general points are unchanged by normalization,
we have
$$\ep(Y_1,(\mu^*L_P)|_{Y_1};1_Y)=\ep(\overline{T_{M'}}, L_P|_{\overline{T_{M'}}};1_P)=\ep(X_{\pi(P)},L_{\pi(P)};1_{\pi(P)}),$$
and the theorem follows.

The last statement is clear since $\ep(X_{\pi(P)},L_{\pi(P)};1_{\pi(P)})=\deg L_{\pi(P)}=|\pi(P)|$ holds when $r=1$.
\end{proof}

\begin{rem}\label{rem for mainprop}
If $\ep(X_{\pi(P)},L_{\pi(P)};1_{\pi(P)}) \leq \ep(X_{P(u')},L_{P(u')};1_{P(u')})$
holds for some $\pi$ and $u'$ in Theorem \ref{estimation by projection},
we have $\epp= \ep(X_{\pi(P)},L_{\pi(P)};1_{\pi(P)})$.
As we will see in the following examples,
we can sometimes compute $\epp$ explicitly by finding such $\pi$ and $u'$.
\end{rem}

\begin{eg}\label{ex_P1 cross P1}
Set $P=\conv((0,0),(a,0),(0,b),(a,b)) \subset \R^2$ for $a \leq b \in \zp$.
We apply Theorem \ref{estimation by projection} to the projection $\pi: \R^2 \rarw \R$ onto the first factor and $u'=0 \in \R$.
Then $\ep(X_P,L_P;1_P)= \ep(X_{\pi(P)},L_{\pi(P)};1_{\pi(P)})=a$ by Remark \ref{rem for mainprop}
because $|\pi(P)| =a \leq b= \ep(X_{P(0)},L_{P(0)};1_{P(0)})$.
Note that $(X_P,L_P)=(\proj^1 \times \proj^1,\calo(a,b))$.
\end{eg}

\begin{eg}\label{ex_toric_fano_3fold}
There are 18 smooth toric Fano 3-folds (cf.\ \cite{Ba}, \cite{WW}).
For an integral polytope $P \subset \R^3$ such that $X_P$ is a smooth Fano 3-fold and $L_P=-K_{X_P}$,
we can easily find a projection $\pi : \R^3 \rarw \R $ and $u' \in \Q$ as in Remark \ref{rem for mainprop}
and compute $\epp$.
As a consequence,
we have
\[
\ep(X_P,L_P;1_P)=
\left\{
\begin{array}{cl}
4 & \mbox{if $X_P=\proj^3 $}\\
3 & \mbox{if $X_P=\proj_{\proj^1}(\calo_{\proj^1} \oplus \calo_{\proj^1} \oplus \calo_{\proj^1}(1))$ }\\
2 & \mbox{otherwise.}
\end{array}
\right.
\]
\end{eg}

\begin{eg}\label{ex of toric at maximal orbit}
We can generalize the estimation in Example \ref{example1}.
For rational numbers $a_1,\ldots,a_n \geq 0$,
set $P=\conv(e_1,\ldots,e_n, -\sum_{i=1}^n a_i e_i) \subset \R^n$.
Then it holds that
\begin{align}
\tag{$\dagger$}
 \epp \geq \min_{1 \leq i \leq n} \dfrac{a_i + \cdots + a_n + 1}{a_{i+1} + \cdots + a_n +1}. 
\end{align}

We show $(\dagger)$ by induction of $n$.
When $n=1$,
$\epp=|P|=a_1+1 $.
Thus $(\dagger)$ holds.

Assume $(\dagger)$ holds for $n-1$.
We apply Theorem \ref{estimation by projection}
to the projection $\pi : \R^n \rarw \R$ onto the $n$-th factor, $P$, and $ 0 \in \pi(P) \cap \Q$.
Then $P(0)=\pi^{-1}(0) \cap P = \conv (e_1,\ldots,e_{n-1}, -1/(a_n+1) \sum_{i=1}^{n-1} a_i e_i)$
and $\pi(P)=[-a_n,1] \subset \R .$
By induction hypothesis,
\begin{eqnarray*}
\ep(X_{P(0)},L_{P(0)};1_{P(0)}) &\geq&
\min_{1 \leq i \leq n-1}  \dfrac{a_i/(a_n+1) + \cdots + a_{n-1}/(a_n+1) + 1}{a_{i+1}/(a_n+1) + \cdots + a_{n-1}/(a_n+1) +1}\\
&=& \min_{1 \leq i \leq n-1} \dfrac{a_i + \cdots + a_n + 1}{a_{i+1} + \cdots + a_n +1} 
\end{eqnarray*}
holds.
By Theorem \ref{estimation by projection},
it follows that
\begin{eqnarray*}
\epp &\geq& \min \{ \ep(X_{\pi(P)},L_{\pi(P)};1_{\pi(P)}) , \ep(X_{P(0)},L_{P(0)};1_{P(0)}) \} \\
 &\geq & \min \Bigl\{ a_n+1, \min_{1 \leq i \leq n-1} \dfrac{a_i + \cdots + a_n + 1}{a_{i+1} + \cdots + a_n +1}  \Bigr\} \\
 &=&  \min_{1 \leq i \leq n} \dfrac{a_i + \cdots + a_n + 1}{a_{i+1} + \cdots + a_n +1}.
\end{eqnarray*}
We will use this lower bound in Section \ref{non-toric}.
\end{eg}

\subsection{At a point in any orbit}
Next, we consider the Seshadri constant on a toric variety at a point
not contained in the maximal orbit.
For a polytope $\sigma \subset \mr$,
we denote by $\sigma-\sigma$ the Minkowski sum of $\sigma$ and  $-\sigma$.
The linear space spanned by $\sigma-\sigma$
is denoted by $\R(\sigma-\sigma)$.

\begin{defn}\label{def of s}
Let $P$ be an integral polytope of dimension $n$ in $\mr$,
and $v$ a vertex of $P$.
We define $s(P; v)$ to be the minimum length of an edge of $P$ containing the vertex $v$;
in symbols, we have
$$s(P;v)=\min \{ \, |\tau|_{M_{\tau}} \, | \, v \prec \tau \prec P, \dim \tau =1 \} \in \zp,$$
where $M_{\tau}=\R(\tau-\tau) \cap M$
and we consider $\tau$ as a subset in $(M_{\tau})_{\R}=\R(\tau - \tau)$ by a parallel translation.
If $M=\{0\}$, we set $s(P;v)= + \infty$ for $P=v=\{0\}$.
\end{defn}

Let $\sigma$ be a face of an $n$-dimensional integral polytope $P \subset \mr$.
Let $\pi:\mr \rarw \mr/\R(\sigma-\sigma)$ be the natural projection and 
set $M'=\pi(M), P'=\pi(P)$, and $v'=\pi(\sigma)$. 
Note that $P'$ is an integral polytope in $M'_{\R}=\mr/\R(\sigma-\sigma)$ and $v'$ is a vertex of $P'$.
The following proposition says that
we can reduce the study of $\ep(X_P,L_P;p)$ for $p \in O_{\sigma}$
to that of $\ep(X_{\sigma},L_{\sigma}; 1_{\sigma})$.
We can regard $\sigma$ as an integral polytope
in $(\R(\sigma - \sigma) \cap M)_{\R}=\R(\sigma - \sigma)$ by a parallel translation,
hence we can consider the polarized toric variety $(X_{\sigma},L_{\sigma})$.

\begin{prop}\label{sc at any point}
Let $\sigma$ be a face of an $n$-dimensional integral polytope $P \subset \mr$.
Set $\pi:\mr \rightarrow \mr/\R(\sigma-\sigma),P'=\pi(P)$, and $v'=\pi(\sigma)$
as above. 
Then,
$$\ep(X_P,L_P; p) = \min \{\ep(X_{\sigma},L_{\sigma}; 1_{\sigma}), s(P';v')\}$$
holds for any $p \in O_{\sigma}$.
\end{prop}

\begin{proof}
We may assume $0 \in \sigma$,
thus $v'=0$ in $M'=\pi(M)$.

Firstly, we show $\ep(X_P,L_P;p) \leq \min \{\ep(X_{\sigma},L_{\sigma};p), s(P';v')\}.$
The torus action implies that $\ep(X_{\sigma},L_{\sigma};p)=\ep(X_{\sigma},L_{\sigma}; 1_{\sigma}) $.
Since $L_P|_{X_{\sigma}} = L_{\sigma}$,
the inequality
\begin{align}
\tag{$\flat$}
\ep(X_P,L_P;p) \leq \ep(X_{\sigma},L_{\sigma};p)
\end{align}
is clear
from the natural inclusion $X_{\sigma} \hookrightarrow X_P$.
By the definition of $\pi$, there is a natural one-to-one correspondence between 
$\Xi:=\{\tau \, | \, \sigma \prec \tau \prec  P, \, \dim\tau=\dim\sigma+1\}$ 
and $\Xi' :=\{\tau'  \, | \, v' \prec \tau' \prec P', \, \dim\tau'=1\}$
by sending $\tau \in \Xi$ to $\pi(\tau) \in \Xi'$.
Fix $\tau' \in \Xi'$ and let $\tau \in \Xi$ be the corresponding face of $P$.
Then by Theorem \ref{estimation by projection}, 
$\ep(X_{\tau},L_{\tau};q) \leq |\tau'|$ holds for any $q \in O_{\tau}$.
Since $\codim (X_{\sigma},X_{\tau}) =1$ and $X_{\tau}$ is normal,
$X_{\tau}$ is smooth at $p$.
Therefore by the lower semicontinuity of Seshadri constants (see \cite[Example 5.1.11]{La2}),
it holds that $\ep(X_P,L_P;p) \leq \ep(X_{\tau},L_{\tau};p) \leq \ep(X_{\tau},L_{\tau};q) \leq |\tau'|$.
By the definition of $s(P';v')$, we have
\begin{align}
\tag{$\natural$}
\ep(X_P,L_P;p) \leq \min_{\tau' \in \Xi'} |\tau'| =s(P';v').
\end{align}
From $(\flat)$ and $(\natural)$,
it follows that $\ep(X_P,L_P;p) \leq \min \{\ep(X_{\sigma},L_{\sigma};p), s(P';v')\}$.

\vspace{3mm}
Next we show the opposite inequality.
Let $C$ be a curve on $X_P$ containing $p$.
It is enough to show
\begin{align}
\tag{$\sharp$}
C.L_P \geq \mult_p(C) \cdot \min \{\ep(X_{\sigma},L_{\sigma};p), s(P';v')\}.
\end{align}
\\
Case 1. $C \subset X_{\sigma}$.

In this case, $C.L_P \geq \mult_p(C) \cdot \ep(X_{\sigma},L_{\sigma};p)$ is clear by the definition of Seshadri constants,
thus $(\sharp)$ holds.
\\
\\
Case 2. $C \not \subset X_{\sigma}$.

We use the following claim.
\begin{clm}\label{claim2}
In this case,
there exist $\tau' \in \Xi'$ and an effective divisor $D \in |L_P \otimes \mathfrak{m}_{p}^{|\tau'|}|$ on $X_P$
such that $C \not \subset \Supp D $.
\end{clm}

\begin{proof}[{Proof of Claim \ref{claim2}}]
For $\tau' \in \Xi'$, let $v'_{\tau'}$ be the vertex of $\tau'$ different from $v'$.
If $\tau \in \Xi$ corresponds to $\tau'$, 
there exists a vertex $v_{\tau}$ of $\tau$ such that $\pi(v_{\tau})=v'_{\tau'}$.
Many vertices of $\tau$ may satisfy this condition,
but we choose one of them.
Let $x^{v_{\tau}} \in H^0(X_P,L_P)$ be the section corresponding to $v_{\tau}$,
and $D_{v_{\tau}} \in |L_P|$ the corresponding effective divisor on $X_P$.
Since $\Supp D_{v_{\tau}}= \bigcup_{v_{\tau} \notin \rho \prec P} X_{\rho}$, 
we have $$\bigcap_{\tau \in \Xi} \Supp D_{v_{\tau}} 
= \bigcup_{v_{\tau} \notin \rho \ \text{for} ^{\forall}\tau \in \Xi} X_{\rho}.$$
By the choices of $v_{\tau}$, $\sigma$ does not contain any $v_{\tau}$.
If $\rho \succ \sigma$ and $\rho \neq \sigma$,
then $\rho $ contains some $\tau \in \Xi $,
hence $v_{\tau} \in \tau \subset \rho$.
Consequently, it holds that 
$$X_{\sigma} \subset \bigcap_{\tau \in \Xi} \Supp D_{v_{\tau}} 
\subset X_{\sigma} \cup \bigcup_{\sigma \not\prec \rho \prec P} X_{\rho}.$$
Since $\bigcup_{\sigma \not\prec \rho \prec P} X_{\rho}$ is a closed set not containing $p$,
$\bigcap_{\tau \in \Xi} \Supp D_{v_{\tau}}$ coincides with $X_{\sigma}$ around $p$.
Now $C$ contains $p$ and is not contained in $X_{\sigma}$ by assumption, 
thus $C$ is not contained in $\bigcap_{\tau \in \Xi} \Supp D_{v_{\tau}}$.
Hence we can choose $\tau_0 \in \Xi$ such that $\Supp D_{v_{\tau_0}}$ does not contain $C$.
Let $\tau'_{0} \in \Xi'$ be the corresponding face,
and set $e'=|\tau'_0|^{-1} v'_{\tau'_0} \in M'$.
Since $v'=0$,
$\tau'_0$ is the convex hull of $v'$ and $v'_{\tau'_0}$.
Then $e'$ is the generator of $\R(\tau'_{0}-\tau'_{0}) \cap M'=\R \tau'_{0} \cap M' \cong \Z$ contained in $\tau'_{0}$.
Fix $e \in M \cap \pi^{-1}(e')$. 
Since $v'_{\tau'_{0}} = |\tau'_{0}|e'$,
$u :=v_{\tau_0} - |\tau'_{0}|e$ is contained in $\pi^{-1}(0) \cap M =\R(\sigma-\sigma) \cap M$.
This means $x^{v_{\tau_0}} = x^u \cdot (x^e)^{|\tau'_{0}|}$ is contained in $H^0(X_P,L_P \otimes \mathfrak{m}_{p}^{|\tau'|})$, 
hence this $\tau'_{0}$ and $D_{v_{\tau_0}}$ satisfy the condition in the claim.
\end{proof}

If there exists such a divisor $D$ in Claim \ref{claim2},
$(\sharp)$ holds since
\[ C.L_P = C.D \geq \mult_p(C) \cdot | \tau ' | \geq \mult_p(C) \cdot s(P';v') .\]
Thus the proof of the proposition is completed.
\end{proof}

\begin{rem}\label{at torus inv points}
For a vertex $v \prec P$,
we have $\ep(X_P,L_P;p) = s(P;v)$ for the torus invariant point $p=O_v$ by Proposition \ref{sc at any point}.
When $X_P$ is smooth,
this is proved by \cite[Corollary 4.2.2]{BDH+}.
\end{rem}

\begin{eg}
Let $P$ be as in Example \ref{example1}.
For any 1-dimensional face $\sigma$ of $P$,
$s(P',v')=|P'|=3$ and $\ep(X_{\sigma},L_{\sigma};1_{\sigma}) =1$.
(We use notations in Proposition \ref{sc at any point} here.)
Thus, we obtain $\ep(X_P,L_P;p)=\min \{1,3 \}=1$ for $p \in O_{\sigma}$.
For any vertex $v \prec P$ and $p=O_v$,
$\ep(X_P,L_P;p)=s(P;v)=1$ by Remark \ref{at torus inv points}.
Thus we have $\ep(X_P,L_P;p)=1$ for $p \in X_P \setminus O_P$.
\end{eg}

\begin{eg}
For an integral polytope $P \subset \R^3$ such that $X_P$ is a smooth Fano 3-fold and $L_P=-K_{X_P}$,
we can easily compute $\ep(X_P,L_P;p)$ for any $p$
by Theorem \ref{estimation by projection} and Proposition \ref{sc at any point}.
As a consequence,
we know $\ep(X_P,L_P;p) \in \{1,2,3,4\}$ for such $P$ and any $p \in X_P$.
\end{eg}

\section{Seshadri constants and toric degenerations}\label{non-toric}

In the previous section,
we study the Seshadri constants on toric varieties.
In this section, we investigate non-toric cases by using toric degenerations.

\begin{defn}\label{sc at very gene}
Let $L$ be a nef $\R$-divisor on a projective variety $X$,
and $\mb=(m_1,\ldots,m_r) \in \r+^r \setminus 0$ for $r >0$.
We say $L(\mb)$ or $L(m_1,\ldots,m_r)$ is \textit{nef} (resp.\ \textit{ample}) if so is
$$\mu^{*} L - \sum_{i=1}^{r} m_i E_i,$$
where $p_1,\ldots,p_r$ are very general $r$ points on $X$,
$\mu : \widetilde{X} \rarw X$ is the blowing up at $p_1,\ldots,p_r$,
and $E_i$ is the exceptional divisor over $p_i$.
In other words,
$L(\mb)$ is nef
if and only if $\ep(X,L;\mb) \geq 1$.
\end{defn}

\begin{rem}\label{suffice for one choice}
To show the nefness of $L(\mb)$,
it is enough to show $ \mu^{*} L - \sum_{i=1}^{r} m_i E_i$ is nef for \textit{one} choice of $p_1,\ldots,p_r$.
This follows from the openness of the ampleness condition as \cite[Lemma.6.1.A]{Bi}.
\end{rem}

By using degenerations, we can show the nefness (resp.\ ampleness) of a divisor
from the nefness (resp.\ ampleness) of other divisors.
The following lemma is a straightforward generalization of Theorem 2.A in \cite{Bi},
so we leave the proof to the reader.

\begin{lem}\label{degeneration}
Let $f:\mathcal{X} \rarw T$ be a flat projective morphism over a smooth variety $T$
with reduced and irreducible general fibers,
and $\mathcal{L}$ an $f$-nef (resp.\ $f$-ample) divisor on $\mathcal{X}$.
Let $X_t=f^{-1}(t)$ be the scheme theoretic fiber of $f$, 
$L_t =\mathcal{L}|_{X_t}$ for $t \in T$.
Assume that $Y_i$ for $ 1 \leq i \leq r $ are irreducible components of the central fiber $X_0$ where $ 0 \in T$ 
with the reduced structures (other components may exist).
If we assume the following:
\begin{itemize}
\item[(i)] $X_0$ is reduced at the generic point of $Y_i$ for any $i$, 
\item[(ii)] There exist $k_i \in \N$ and $\mb^{(i)} = (m^{(i)}_1,\ldots,m^{(i)}_{k_i}) \in \r+^{k_i} \setminus 0$ for
$1 \leq i \leq r $ such that
$\mathcal{L}|_{Y_i}(\mb^{(i)})$ is nef (resp.\ ample) for any $i$,
\end{itemize}
then $L_t(\mb^{(1)},\ldots,\mb^{(r)})$ is nef (resp.\ ample)
for very general $t \in T$. 
\end{lem}

\vspace{1mm}
Lemma \ref{degeneration}
tells us a strategy to obtain lower bounds for (multi-point) Seshadri constants at very general points:
\textit{Finding degenerations to (unions of) polarized varieties whose Seshadri constants are more computable,
such as toric varieties.}
For example,
if the central fiber $X_0$ is reduced and irreducible,
and the normalization of $(X_0,\mathcal{L}|_{X_0})$ is isomorphic to
$(X_P,L_P)$ for an integral polytope $P \subset \mr$,
then we have $\ep(X_t,L_t;1) \geq \ep(X_0,\mathcal{L}|_{X_0};1)= \ep(X_P,L_P;1)$ for very general $t$.

In the rest of this paper,
we give some explicit estimations or computations of Seshadri constants by using this strategy.

\subsection{Hypersurfaces and complete intersections in projective spaces}

In this subsection,
we study Seshadri constants on hypersurfaces or complete intersections in projective spaces.
For positive integers $d_1,\ldots,d_k$ and $n$,
we denote by $X^n_{d_1,\ldots,d_k}$ a very general complete intersection
of hypersurfaces of degrees $d_1,\ldots,d_k$ in $\proj^{n+k}$.

Firstly,
we estimate $\ep(X,\calo(1);1)$ for a very general complete intersection $X$.

\begin{prop}\label{a point in c.i.}
Let $d_1,\ldots,d_k$ and $n$ be positive integers.
Suppose that there exist a positive integer $c$ and natural numbers $l_1,\ldots,l_k$ such that
$\sum_{j=1}^k l_j=n$ and $d_j \geq c^{\, l_j} $ hold for any $1 \leq  j \leq k$.
We have $\ep(X^n_{d_1,\ldots,d_k},\calo(1);1) \geq c$.

In particular,
$\ep(X^n_{d},\calo(1);1) \geq \lfloor \sqrt[n]{d} \rfloor$ holds for any $d \in \zp.$
\end{prop}

\begin{proof}
We prove this proposition
by 3 steps.
\vspace{2mm}
\ \\
\textbf{Step 1.}
Firstly,
we find a toric variety
which is a complete intersection
of hypersurfaces of degrees $d_1,\ldots,d_k$ in $\proj^{n+k}$.
Let $d^{(i)}_j $ be natural numbers for $1 \leq i \leq n, 1 \leq j \leq k$ such that
$1+\sum_{i=1}^n d^{(i)}_j=d_j$ holds for any $j$.
Let $X \subset \proj^{n+k}$ be the subvariety defined by the following homogeneous polynomials
\begin{eqnarray*}
T_0^{d_1} \! \! \! \! \! &-& \! \! \! \! T_1^{d^{(1)}_1} T_2^{d^{(2)}_1} \cdots T_{n}^{d^{(n)}_1} T_{n+1}\\
T_{n+1}^{d_2} \! \! \! \! \! &-&  \! \! \! \! T_1^{d^{(1)}_2} T_2^{d^{(2)}_2} \cdots T_{n}^{d^{(n)}_2} T_{n+2}\\
 & & \ \ \ \ \ \ \ \ \  \vdots\\
T_{n+k-1}^{d_k} \! \! \! \! \! &-& \! \! \! \! T_1^{d^{(1)}_k} T_2^{d^{(2)}_k} \cdots T_{n}^{d^{(n)}_k} T_{n+k},
\end{eqnarray*}
where $T_0, \ldots,T_{n+k}$ are the homogeneous coordinates on $\proj^{n+k}$.

Let $P$ be the image of the standard simplex $ \conv (0, e_1, \ldots, e_{n+k}) \subset \R^{n+k}$
by the lattice projection
$$ \pi: \R^{n+k}=(\Z^{n+k})_{\R} \rarw (\Z^{n+k} / \Lambda)_{\R}, $$
where $\Lambda$ is the subgroup of $\Z^{n+k}$ spanned by
$ \sum_{i=1}^n d^{(i)}_1 e_i + e_{n+1}$ and $ \sum_{i=1}^n d^{(i)}_j e_i  - d_j e_{n+j-1} +e_{n+j} $ for $2 \leq j \leq k$.
By construction,
the generators of $\Lambda$ correspond to the given system of homogeneous polynomials.
For example,
$ \sum_{i=1}^n d^{(i)}_1 e_i + e_{n+1}$ corresponds to the binomial
$T_0^{d_1} - \ T_1^{d^{(1)}_1} T_2^{d^{(2)}_1} \cdots T_{n}^{d^{(n)}_1} T_{n+1}$.
By Lemma \ref{orbit},
$(X,\calo(1))$ is a toric variety
whose normalization is $(X_P,L_P)$.
Since $X^n_{d_1,\ldots,d_k}$ degenerates to $X$,
we have $\ep(X^n_{d_1,\ldots,d_k},\calo(1);1) \geq \ep(X,\calo(1);1)= \ep(X_P,L_P;1) $ by Lemma \ref{degeneration}.
Thus it suffices to show $\ep(X_P,L_P;1) \geq c$
for a suitable choice of $d^{(i)}_j$.
\vspace{1mm}
\ \\
\textbf{Step 2.}
Secondly, we estimate $\ep(X_P,L_P;1) $ by $d^{(i)}_j $.
We denote $\pi(e_l)$ by $[e_l]$ for $1 \leq l \leq n+k$.
Since the coefficient of $e_{n+1}$ in $ \sum_{i=1}^n d^{(i)}_1 e_i + e_{n+1}$
and that of $e_{n+j}$
in $ \sum_{i=1}^n d^{(i)}_j e_i  - d_j e_{n+j-1} +e_{n+j} $ are $1$ for $2 \leq j \leq k$,
we can take $[e_1],\ldots,[e_n]$ as a basis of $\Z^{n+k} / \Lambda$.
Since $[e_{n+1}]=-\sum_{i=1}^n d^{(i)}_1 [e_i] $
and $[e_{n+j}]= -\sum_{i=1}^n d^{(i)}_j [e_i] + d_j [e_{n+j-1}]$,
we can show $ [e_{n+k}]=-\sum_{i=1}^n a^{(i)} e_i$ for
\begin{eqnarray*}
a^{(i)}&=&\sum_{j=1}^k  d^{(i)}_j d_{j+1} \cdots d_k\\
 &=&  d^{(i)}_1 d_2 \cdots d_k + \cdots +  d^{(i)}_{k-1} d_k+ d^{(i)}_k.
\end{eqnarray*}
It is easy to see that $P$ is the convex hull of $[e_1], \ldots, [e_{n}]$, and $[e_{n+k}]$.
By Example \ref{ex of toric at maximal orbit},
we have
$$\ep(X_P,L_P;1) \geq \min_{1 \leq i \leq n}  b^{(i)} / b^{(i+1)} , $$
where $b^{(i)}=a^{(i)}+a^{(i+1)}+\cdots + a^{(n)} +1$ for $1 \leq i \leq n$ and $b^{(n+1)}=1$.
\vspace{2mm}
\ \\
\textbf{Step 3.}
The complete intersection $ X^n_{d_1,\ldots,d_k} = \bigcap_j X^{n+k-1}_{d_j} $
degenerates to $\bigcap_j (X^{n+k-1}_{c^{\, l_j}} \cup X^{n+k-1}_{d_j-c^{\, l_j}}). $
Since $X^{n+k-1}_{c^{\, l_j}}$ and $ X^{n+k-1}_{d_j-c^{\, l_j}}$ are very general,
$ X^n_{c^{\, l_1},\ldots, c^{\, l_k}} = \bigcap_j X^{n+k-1}_{c^{\, l_j}}$
is an irreducible component of $\bigcap_j (X^{n+k-1}_{c^{\, l_j}} \cup X^{n+k-1}_{d_j-c^{\, l_j}}). $
By applying Lemma \ref{degeneration} to this degeneration,
we have $$ \ep(X^n_{d_1,\ldots,d_k},\calo(1);1 ) \geq \ep(X^n_{c^{\, l_1},\ldots, c^{\, l_k}} ,\calo(1);1 ).$$
Thus it is enough to show this proposition for $d_j=c^{\, l_j}$.

Let us define $d^{(i)}_j$ for $d_j=c^{\, l_j}$
such that $b^{(i)} /b^{(i+1)} = c$ for any $1 \leq i \leq n$.
Set
\[
d^{(i)}_j=
\left\{
\begin{array}{cc}
 (c-1)c^{\, h_j-i} &  \mbox{if $h_{j-1} < i \leq h_j $}  \ \\
 0 &  \mbox{otherwise} ,
\end{array}
\right .
\]
where $h_j=l_1 + \cdots + l_j$ for $1 \leq j \leq k$ and $h_{0}=0$.
In particular, we have $h_k= \sum_j l_j=n $ and $ \sum_i d^{(i)}_j = c^{\, l_j} -1= d_j -1$.
For each
$ h_{j-1} < i \leq h_j $,
we have
$$a^{(i)}=d_{j+1} \cdots d_k (c-1) c^{\, h_{j}-i}= (c-1)c^{\, n-i}.$$
Hence $b^{(i)}= c^{n+1-i} $ holds for any $i$
and we obtain $b^{(i)} /b^{(i+1)} = c^{n+1-i} / c^{n-i} = c$,
which proves this proposition.
\end{proof}

If we choose $d^{(i)}_j$ carefully,
we may obtain a better estimation than that of Proposition \ref{a point in c.i.}.
We use notations in the proof of Proposition \ref{a point in c.i.} in the following examples.

\begin{eg}\label{examples of c.i. Fano at a point}
Let $2 \leq d_1 \leq  \ldots \leq  d_k $ be positive integers
such that $\sum_j d_j \leq n+k $.
Then $X^n_{d_1,\ldots,d_k}$ is a Fano $n$-fold such that
$-K_{X^n_{d_1,\ldots,d_k}}=\calo(n+k+1 - \sum_j d_j)$. 
If $\sum_j d_j < n+k  $,
it is known that $X^n_{d_1,\ldots,d_k}$ is covered by lines (cf.\ \cite[Proposition 2.13]{Deb}).
Hence we have $\ep(X^n_{d_1,\ldots,d_k},\calo(1);1 )=1$.

\vspace{2mm}
If $\sum_j d_j=n+k$,
then $X^n_{d_1,\ldots,d_k}$ is a Fano $n$-fold such that
$-K_{X^n_{d_1,\ldots,d_k}}=\calo(1)$.
We can show $\ep(X^n_{d_1,\ldots,d_k},\calo(1);1)=d_k / (d_k-1) $ as follows.

\vspace{1mm}
We define $d^{(i)}_j$ by
\[
d^{(i)}_j=
\left\{
\begin{array}{cc}
 1 &  \mbox{if $h'_{j-1} < i \leq h'_j $}  \ \\
 0 &  \mbox{otherwise} ,
\end{array}
\right .
\]
where $h'_{j}=(d_{1}-1)+\cdots+(d_j-1) .$
Note $h'_{k}=\sum_{j=1}^k (d_j-1)=\sum_j d_j -k=n$.
Then we have $a^{(i)}=d_{j+1}\cdots d_k$ and $ b^{(i)}=d_{j+1}\cdots d_k (h'_{j}+2-i) $ for $ h'_{j-1} < i \leq h'_j .$
By Steps 1 and 2 in the proof of Proposition \ref{a point in c.i.},
we have
$$ \ep(X^n_{d_1,\ldots,d_k},\calo(1);1) \geq \min_{1 \leq i \leq n} \dfrac{b^{(i)}}{b^{(i+1)}} = \min_{1 \leq j \leq k} \dfrac{d_j}{d_j-1} = \frac{d_k}{d_k -1}.$$

\vspace{1mm}
Next,
we show $ \ep(X^n_{d_1,\ldots,d_k},\calo(1);1) \leq d_k / (d_k -1)$
by finding a curve $C \subset X^n_{d_1,\ldots,d_k}$ such that $C.\calo(1) / \mult_p(C) =d_k / (d_k -1)$
for any very general point $p \in X^n_{d_1,\ldots,d_k}$.

Let $F_1,\ldots,F_k$ be homogeneous polynomials in $\C[T_0,\ldots,T_{n+k}]$ of degrees $d_1,\ldots,d_k$ respectively
such that $X^n_{d_1,\ldots,d_k}=(F_1 =\cdots=F_k=0)$.
We may assume $p=[1:0:\cdots:0] \in \proj^{n+k}$.
Then there exist homogeneous polynomials $F_j^i \in \C[T_1,\ldots,T_{n+k}]$
such that $\deg F_j^i=i$ and $F_j=\sum_{i=1}^{d_j} T_0^{d_j-i} F_j^i $.
Let $D_j$ and $D_j^i \subset \proj^{n+k}$ be the hypersurfaces defined by $F_j$ and $F_j^i $ respectively.
Consider
$$\bigcap_{i=1}^{d_j} D_j^i = (F_j^1=\cdots=F_j^{d_j}=0) \subset D_j$$
for $1 \leq j \leq k-1$, and
$$ \bigcap_{i=1}^{d_k-2} D_k^i \cap (T_0 F_k^{d_k-1}+F_k^{d_k}=0)
= (F_k^1=\cdots=F_k^{d_k-2}=T_0 F_k^{d_k-1}+F_k^{d_k}=0) \subset D_k.$$
The $F_j^i$ are very general because $X^n_{d_1,\ldots,d_k}$ and $p$ are very general.
Hence
$$C:=\bigcap_{j=1}^{k-1} \bigcap_{i=1}^{d_j} D_j^i \cap \left(\bigcap_{i=1}^{d_k-2} D_k^i \cap (T_0 F_k^{d_k-1}+F_k^{d_k}=0)\right)$$
is a complete intersection curve in $\proj^{n+k}$,
and $p \in C \subset \cap_{j=1}^k D_j=X^n_{d_1,\ldots,d_k}$.
By definition,
$$\deg C= \prod_{j=1}^{k-1} d_j ! \cdot (d_k-2)! \cdot d_k,
\ \mult_p (C) = \prod_{j=1}^{k-1} d_j ! \cdot (d_k-1)! $$
hold.
Thus we have
$$\ep(X^n_{d_1,\ldots,d_k},\calo(1);1)=\ep(X^n_{d_1,\ldots,d_k},\calo(1);p) \leq \frac{\deg C}{\mult_p (C)} = \frac{d_k}{d_k-1}.$$
Therefore it follows that $\ep(X^n_{d_1,\ldots,d_k},\calo(1);1)=d_k / (d_k-1)$.

For example,
we have
$$\ep(X^3_4,\calo(1);1)  = 4/3, \
\ep(X^3_{2,3},\calo(1);1) = 3/2, \
\ep(X^3_{2,2,2},\calo(1);1) =  2$$
when $n=3$.
\end{eg}

\begin{eg}\label{examples of hypersurf at a point}
When $k=1$,
we write $d=d_1,d^{(i)}=d_1^{(i)}$ for simplicity.
Then,
$a^{(i)}=d^{(i)}$ for any $i$.
Thus we have
$$\ep(X_d^{n},\calo(1);1) \geq \min_{1 \leq i \leq n} \dfrac{d^{(i)}+\cdots + d^{(n)} + 1}{d^{(i+1)}+\cdots + d^{(n)} + 1 }.$$
In other words,
$$\ep(X_d^{n},\calo(1);1) \geq \min \left\{ \frac{c_n}1, \frac{c_{n-1}}{c_n},\ldots, \frac{c_2}{c_3}, \frac{c_1}{c_2} \right\} $$
holds for any increase sequence of positive integers $1 \leq c_n \leq c_{n-1} \leq \cdots \leq c_1=d.$
For instance,
$\ep(X_{22}^{3},\calo(1);1) \geq \min \{ 3/1, 8/3, 22/8 \} =8/3 $ follows from $1 \leq 3 \leq 8 \leq 22$.

When $n=2$,
set $c_1=d , c_2=\lceil \sqrt{d} \ \rceil $.
Then 
$\ep(X_d^2,\calo(1);1) \geq \min \{ \lceil \sqrt{d} \ \rceil, d / \lceil \sqrt{d} \ \rceil \}= d / \lceil \sqrt{d} \ \rceil$ holds.
From this and Proposition \ref{a point in c.i.}, we have
$$\ep(X_d^2,\calo(1);1) \geq \max \{ \lfloor \sqrt{d} \ \rfloor, d /  \lceil \sqrt{d} \ \rceil \} .$$
When $d \geq 4 $,
$\ep(X_d^2,\calo(1);1) \geq \lfloor \sqrt{d} \ \rfloor$ follows from Proposition 1 in \cite{St} as well
since $\Pic X = \Z \calo_X(1)$.
But $d /  \lceil \sqrt{d} \ \rceil $ is a new estimation.
For example, $\ep(X_7^2,\calo(1);1) \geq 7/3$ holds.
\end{eg}

\begin{rem}\label{simple_pf_for_hypersurf}
The final part of Proposition \ref{a point in c.i.},
i.e.,
$\ep(X^n_{d},\calo(1);1) \geq \lfloor \sqrt[n]{d} \rfloor$,
can be easily shown without using estimations in toric cases as follows.

By the first statement of Step 3 in the proof of Proposition \ref{a point in c.i.},
we may assume $d=c^{\,n}$ for $c \in \zp$.
We consider the embedding $i : \proj^n \hookrightarrow \proj^N$
defined by $|\calo_{\proj^n}(c)|$.
By a suitable linear projection $\pi : \proj^N  \dashrightarrow \proj^{n+1}$, 
$\pi \circ i : \proj^n \rarw \proj^{n+1}$ is a finite birational morphism onto the image $Y:=\pi \circ i (\proj^n)$.
Thus $\ep(Y,\calo(1);1)=\ep(\proj^n,\calo(c);1)=c$ holds since $(\pi \circ i)^* \calo_Y(1) =\calo_{\proj^n}(c)$.
By definition
$Y$ is a hypersurface of degree $d$ in $\proj^{n+1}$,
hence $X_d^n$ degenerates to $Y$.
Therefore we have $\ep(X_d^n,\calo(1);1) \geq \ep(Y,\calo(1);1)=c=\lfloor \sqrt[n]{d} \rfloor$
by the lower semicontinuity of Seshadri constants.

We also have the following observation:
Seshadri constants at very general points take maximal values for hypersurfaces in projective spaces.

For any polarized variety $(X,L)$ such that $L$ is very ample,
$$\ep(X,L;\mb) \leq \ep(X^n_d,\calo(1);\mb) $$
holds for any $\mb \in \r+^r \setminus 0$,
where $n=\dim X$ and $d=L^n$.
\end{rem}

\vspace{2mm}

Next,
we study multi-point cases.
The following lemma looks like Theorem 2.A in \cite{Bi}.

\begin{lem}\label{multi point in c.i.}
Let $d_1,\ldots,d_k , a, b$, and $n $ be positive integers for $k \in \N$.
We denote by $L_a,  L_b$, and $ L_{a+b}$
the invertible sheaf $\calo(1)$
on $X^n_{d_1,\ldots,d_k,a}, X^n_{d_1,\ldots,d_k,b}$, and $ X^n_{d_1,\ldots,d_k,a+b}$
respectively.
If $L_a(\mb_1) $ and $ L_b(\mb_2) $
are nef (resp.\ ample) for  $\mb_1 \in \r+^{r_1} \setminus 0$ and $\mb_2 \in \r+^{r_2} \setminus 0$,
then $L_{a+b}(\mb_1,\mb_2)  $ is also nef (resp.\ ample).
\end{lem}

\begin{proof}
A very general hypersurface $ X^{n+k}_{a+b} \subset \proj^{n+k+1}$ of degree $a+b$ degenerates to
the union $X^{n+k}_a \cup X^{n+k}_b $ of hypersurfaces of degrees $a$ and $b$.
Thus $X^n_{d_1,\ldots,d_k,a+b}=X^{n+1}_{d_1,\ldots,d_k} \cap X^{n+k}_{a+b}$
degenerates to 
$X^{n+1}_{d_1,\ldots,d_k} \cap  (X^{n+k}_a \cup X^{n+k}_b) 
=X^{n}_{d_1,\ldots,d_k,a} \cup X^{n}_{d_1,\ldots,d_k,b}  $.
Applying Lemma \ref{degeneration} to this degeneration,
this lemma follows.
\end{proof}

As a corollary of Proposition \ref{a point in c.i.} and Lemma \ref{degeneration} or Lemma \ref{multi point in c.i.},
we obtain estimates for multi-point Seshadri constants
on hypersurfaces in projective spaces.

\begin{thm}[=Theorem \ref{intro thm 5}]\label{sc of hypersurfaces}
For a very general hypersurface $X^n_d$ of degree $d$ in $\proj^{n+1}$,
we have
$$\lfloor \sqrt[n]{d / (m_1^n+\cdots + m_r^n)} \rfloor \leq \ep(X^n_{d},\calo(1);\mb) \leq \sqrt[n]{d / (m_1^n+\cdots + m_r^n)} $$
for any $\mb=(m_1,\ldots,m_r)  \in \N^r \setminus 0.$
\end{thm}

\begin{proof}
The second inequality is clear since $\calo_{X^n_d}(1)^n=d$.
Thus it remains to prove the first inequality.
For simplicity, set $c=\lfloor \sqrt[n]{d / (m_1^n+\cdots + m_r^n)} \rfloor$. 
Let $d_1,\ldots,d_r$ be natural numbers such that $d=d_1+\ldots + d_r$
and $d_i \geq (c m_i)^n $.
We can choose such $d_i$
because $d \geq \sum_i (c m_i)^n$. 
By Proposition \ref{a point in c.i.},
$\ep(X^n_{d_i},\calo(1);1) \geq c m_i$ holds.
Note that $c m_i$ is an integer.
Since $X^n_d$ degenerates to $\bigcup_{i=1}^r X^n_{d_i}$,
we can apply Lemma \ref{degeneration}
and we know $L_d(c m_1,\ldots,c m_r)$ is nef,
where $L_{d}$ is the invertible sheaf $\calo(1)$ on $X^n_d$.
Hence $\ep(X^n_d,\calo(1);\mb) \geq c$ holds.
\end{proof}

\begin{rem}
Theorem \ref{sc of hypersurfaces} tells us an interesting property
of the nef cone $\Nef(\widetilde{X})$ of the blowing up $\pi : \widetilde{X} \rarw X$
along very general points $p_1,\ldots,p_r$ in $X=X^n_d$.
Set $E_i=\pi^{-1}(p_i)$ be the exceptional divisor over $p_i$.

We identify $\R^r$ with the affine hyperplane $\pi^* \calo_X(1) + \bigoplus_{i=1}^r \R E_i \subset N^1(\widetilde{X})$
by sending $\mb=(m_1,\ldots,m_r) \subset \R^r$ to $\pi^* \calo_X(1) - \sum_{i=1}^r m_i E_i $.
Under this identification,
$\Nef(\widetilde{X}) \cap \R^r$ is equal to $ \{ \, \mb \in \R^r \, | \, \pi^* \calo_X(1) - \sum_{i=1}^r m_i E_i \ \text{is nef} \, \}$.
Since the self intersection number of a nef line bundle is non-negative,
we have $\Nef(\widetilde{X}) \cap \R^r \subset \Delta^n_d:= \{ \, \mb \in \r+^r \, | \,  \sum_{i=1}^r m_i^n \leq d \} $.

Theorem \ref{sc of hypersurfaces} can be rephrased as
$$ \conv ( \N^r \cap \Delta^n_d ) \subset  \Nef(\widetilde{X}) \cap \R^r \subset  \Delta^n_d.$$
Since $\Delta^n_d= \sqrt[n]{d} \, \Delta^n_1 $,
$\Nef(\widetilde{X}) \cap \R^r$ is ``close to'' $\Delta^n_d$ for $d \gg n,r$.
\end{rem}

\subsection{Fano 3-folds with Picard number one}

\vspace{1mm}
In this subsection,
we estimate Seshadri constants on a smooth Fano 3-fold $X$ with Picard number $1$.
In other words,
$X$ is a smooth projective variety of dimension $3$ such that $-K_X$ is ample and $\Pic X \cong \Z$.
The \textit{index} of $X$ is the positive integer $r$ such that $-K_X=rH$,
where $H \in \Pic X$ is the ample generator.

Toric degenerations of Fano 3-folds are studied by many authors.
Small toric degenerations of Fano 3-folds are treated by \cite{Gal},
and \cite{CI} investigated complete intersection cases in (weighted) projective spaces
and homogeneous spaces.
In \cite{ILP},
Ilten, Lewis, and Przyjalkowski studied remaining cases of Fano 3-folds with Picard number $1$.
They showed that
every smooth Fano 3-fold with Picard number 1 has a toric degeneration
and gave an explicit description of the moment polytope of the central fiber.
Most of the degenerations in \cite{ILP} give good lower bounds for Seshadri constants.

\vspace{.5mm}

\begin{eg}\label{(6) in P(1,1,1,1,3)}
Let $X \subset \proj(1,1,1,1,3)$ be a very general hypersurface of degree $6$.
By \cite[First Main Theorem]{ILP},
$(X,\calo(1))$ degenerates to $(X_P,L_P)$ as a $\Q$-polarized variety
for $P:=\conv ( e_1,e_2,e_3,-1/3(e_1+e_2+e_3)) \subset \R^3$.
Thus we have $\ep(X,\calo(1);1) \geq \ep(X_P,L_P;1) \geq 6/5$
by Example \ref{ex of toric at maximal orbit} and Lemma \ref{degeneration}.

Although we can show $\ep(X,\calo(1);1) \leq 6/5$ by the method used in Example \ref{examples of c.i. Fano at a point},
we present a more geometric argument (though both proofs are essentially same).

Fix a very general point $p \in X$.
Define $p' \in X$ by
$\{p,p'\} := \varphi ^{-1} (\varphi(p))$,
where $\varphi:X \rarw \proj^3$ is the double cover defined by $|\calo_X(1)|$. 
Since $\dim H^0(X,\calo_X(3))=21$ and $\dim \calo_X / \m _p ^4=20$,
there exists $S \in |\calo_X(3) \otimes \m _p ^4|$.
Then $\mult_p(S)=4$ because $X$ and $p$ are very general.
It is not hard to see that $S$ does not contain $p'$.
Let $\pi:\widetilde{X} \rarw X$ be the blowing up at $\{p,p'\}$,
and set $E,E'$ to be the exceptional divisors over $p$ and $p'$ respectively.
Let $\widetilde{S} \subset \widetilde{X}$ be the strict transform of $S$,
and set $\psi=\varphi |_{\widetilde{S}} : \widetilde{S} \rarw S$ and
$F=E |_{\widetilde{S}}$.
Then $F^2=-\mult_p(S)=-4 $.
Since $\varphi^* \calo_X(1) -E-E'$ is base point free,
so is $(\varphi^* \calo_X(1) -E-E') |_{\widetilde{S}}=\psi^* \calo_S(1) -F$.
Let $f : \widetilde{S} \rarw \proj^2$ be the morphism defined by $\varphi^* \calo_S(1) -F$.
By $\ep(S,\calo(1);p) \geq \ep(X,\calo(1);p) >1$,
we know $\varphi^* \calo_S(1) -F$ is ample.
Thus $f$ is a finite morphism.
Since $f_*F.\calo_{\proj^2}(1)=F.f^*\calo_{\proj^2}(1)=F.(\varphi^* \calo_S(1) -F)=4$,
we have $f_*F \sim \calo_{\proj^2}(4)$.
Thus $D:=f^*f_*F -F \sim \psi^* \calo_S(4) - 5F$ is an effective divisor.
Hence $\psi^*\calo_S(1)-6/5 F$ is not ample
because $D. (\psi^*\calo_S(1)-6/5 F)=(\psi^* \calo_S(4) - 5F).(\psi^*\calo_S(1)-6/5 F)
=0$.
Thus $\ep(X,\calo(1);1) = \ep(X,\calo(1);p) \leq \ep(S,\calo(1);p) \leq 6/5$ holds
and we have $\ep(X,\calo(1),1)=6/5$.
\end{eg}

It is known that there are 17 families of smooth Fano 3-folds with Picard number $1$.
For each case,
we can compute the Seshadri constant as follows.

\begin{thm}[=Theorem \ref{intro thm 4}]\label{Fano 3-fold}
For each family of smooth Fano 3-folds with Picard number $1$,
$\ep(X,-K_X;1)$ is as in Table \ref{table},
where $X$ is a very general member in the family.
\end{thm}

\begin{proof}
For No.1\,-\,4 in Table \ref{table},
$\ep(X,-K_X;1)$ is computed in Examples \ref{examples of c.i. Fano at a point} and \ref{(6) in P(1,1,1,1,3)}.
(In fact,
degenerations in \cite{ILP} for No.2\,-\,4 give same lower bounds as Example \ref{examples of c.i. Fano at a point}
though some of their degenerations are different from those in Example \ref{examples of c.i. Fano at a point}.)

For No.5\,-\,17,
we can show the following by using toric degenerations.
\begin{align}
\tag{$\ddagger$}
\ep(X,-K_X;1) \geq 
\left\{
\begin{array}{cl}
 2 &  \mbox{for No.5\,-\,15}  \\
 3 &  \mbox{for No.16} \\
 4 &  \mbox{for No.17.}
\end{array}
\right .
\end{align}
Except for No.11,
these lower bounds are obtained by
applying Theorem \ref{estimation by projection} and Lemma \ref{degeneration}
to the degenerations in \cite[First Main Theorem]{ILP}.
In No.11 case,
we consider the following degeneration,
whose construction is essentially same as Proposition \ref{a point in c.i.}.

Let $T_0,T_1,T_2,T_3,T_4$ be weighted homogeneous coordinates on $\proj(1,1,1,2,3)$
with $\deg T_0=\deg T_1 = \deg T_2=1, \deg T_3=2, \deg T_4=3$.
Then $X_0:=\Zeros (T_4^2 - T_1^2 T_2^2 T_3) \subset \proj(1,1,1,2,3)$ is a non-normal toric variety
whose moment polytope is $P= \conv (0,e_1,e_2, -e_1 -e_2 +e_3)$.
A very general hypersurface $X$ in $\proj(1,1,1,2,3)$ of degree $6$
degenerates to $X_0$.
Since $-K_X=\calo_X(2)$,
$(X,-K_X)$ degenerates to $(X_0,\calo_{X_0}(2))$.
Thus $\ep(X,-K_X;1) \geq \ep(X_0,\calo_{X_0}(2);1)= \ep( X_{2P},L_{2P};1) =2$.

Next,
we think about the upper bounds.
For No.5\,-\,10,
it is known that $X$ is covered by conics,
i.e.,
for any general $p \in X$,
there exists a smooth rational curve $C$ containing $p$ such that $C.(-K_X)=2$ (cf.\ \cite[Chapter 4]{IsP}).
Thus $\ep(X,-K_X;1) \leq 2$ in these cases.
For No.11\,-\,15,
$-K_X=2H$ holds for the ample generator $H$.
Assume that $\ep(X,-K_X;1) > 2$,
i.e.,
$-K_{\widetilde{X}}=\mu^*(-K_X)-2E=2(\mu^*H-E)$ is ample
for the blowing up $\mu : \widetilde{X} \rarw X$ at a very general point $p \in X$ and $E=\mu^{-1}(p)$.
Then $\widetilde{X}$ is a Fano 3-fold of index $2$,
i.e.,
a Del Pezzo 3-fold,
and the Picard number is $2$.
By the classification of Del Pezzo manifolds (cf.\ \cite[\S 12.1]{IsP}),
$(-K_{\widetilde{X}})^3=8 (\pi^*H-E)^3$ must be $8 \cdot 6$ or $8 \cdot 7$,
which contradicts $H^3 \leq 5$.
Thus $\ep(X,-K_X;1) \leq 2 $ holds for No.11\,-\,15.
For No.16 and 17,
$X$ is covered by lines
since $X$ is a smooth quadric or $\proj^3$.
Hence $\ep(X,-K_X;1) = r \, \ep(X,H;1) \leq r $ holds for the index $r$
and the ample generator $H$ for No.16 and 17.

Thus the inequalities in $(\ddagger)$ are in fact equalities,
and the proof is completed.
\end{proof}


\begin{thebibliography}{Kol}




\bibitem[Bau]{Bau}
T.~Bauer,
\emph{Seshadri constants and periods of polarized abelian varieties},
with an appendix by the author and Tomasz Szemberg. Math. Ann. \textbf{312} (1998), no. 4, 607--623.



\bibitem[B+]{BDH+}
T.~Bauer, S.~Di Rocco, B.~Harbourne, M.~Kapustka, A.~Knutsen, W.~Syzdek, and T.~Szemberg,
\emph{A primer on Seshadri constants},
Interactions of classical and numerical algebraic geometry, 33--70, Contemp. Math., \textbf{496}, Amer. Math. Soc., Providence, RI, 2009.

\bibitem[Bat]{Ba}
V.V.~Batyrev,
\emph{Toric Fano threefolds},
Izv. Akad. Nauk SSSR Ser. Mat. \textbf{45} (1981), no. 4, 704--717, 927.


\bibitem[Bi]{Bi}
P.~Biran,
\emph{Constructing new ample divisors out of old ones},
Duke Math. J. \textbf{98} (1999), no. 1, 113--135.


\bibitem[CI]{CI}
J.~Christophersen and N.O.~Ilten,
\emph{Degenerations to Unobstructed Fano Stanley-Reisner Schemes},
arXiv:1102.4521.


\bibitem[Deb]{Deb}
O.~Debarre,
\emph{Higher-dimensional algebraic geometry},
Universitext. Springer-Verlag, New York, 2001. xiv+233 pp.


\bibitem[Dem]{Dem}
J.P.~Demailly,
\emph{Singular Hermitian metrics on positive line bundles},
Complex algebraic varieties (Bayreuth, 1990), 87--104, Lecture Notes in Math., \textbf{1507}, Springer, Berlin, 1992.


\bibitem[Di]{Di}
S.~Di Rocco,
\emph{Generation of k-jets on toric varieties},
Math. Z. \textbf{231} (1999), no. 1, 169--188.



\bibitem[Ec]{Ec}
T.~Eckl,
\emph{An asymptotic version of Dumnicki's algorithm for linear systems in $\C \proj^2$},
Geom. Dedicata \textbf{137} (2008), 149--162.

\bibitem[EKL]{EKL}
L.~Ein, O,~K\"{u}chle, and R.~Lazarsfeld,
\emph{Local positivity of ample line bundles},
J. Differential Geom. \textbf{42} (1995), no. 2, 193--219.


\bibitem[Ei]{Ei}
D.~Eisenbud, 
\emph{Commutative Algebra,with a View Toward Algebraic Geometry},
Graduate Texts in Math, no.\textbf{150}, Springer-Verlag, New York, 1995.

\bibitem[Fu]{Fu}
W.~Fulton,
\emph{Introduction to toric varieties},
Annals of Mathematics Studies, \textbf{131}.
Princeton University Press, Princeton, NJ, 1993. xii+157 pp.


\bibitem[Ga]{Gal}
S.~Galkin,
\emph{Small toric degenerations of Fano threefolds},
\url{http://www.mi.ras.ru/~galkin/work/3a.pdf}, (2007).



\bibitem[HK]{HK}
J.-M.~Hwang and J.H.~Keum,
\emph{Seshadri-exceptional foliations},
Math. Ann. \textbf{325} (2003), no. 2, 287--297.

\bibitem[ILP]{ILP}
N.O.~Ilten, J.~Lewis, and V.~Przyjalkowski,
\emph{Toric Degenerations of Fano Threefolds Giving Weak Landau-Ginzburg Models},
arXiv:1102.4664, to appear in J. Algebra.



\bibitem[Is1]{Is1}
V.A.~Iskovskih,
\emph{Fano threefolds. I},
Izv. Akad. Nauk SSSR Ser. Mat. \textbf{41} (1977), no. 3, 516--562, 717. 



\bibitem[Is2]{Is2}
V.A.~Iskovskih,
\emph{Fano threefolds. II},
 Izv. Akad. Nauk SSSR Ser. Mat. \textbf{42} (1978), no. 3, 506--549.


\bibitem[IP]{IsP}
V.A.~Iskovskikh and Yu.G.~Prokhorov,
\emph{Fano varieties. Algebraic geometry, V},
Encyclopaedia Math. Sci., \textbf{47}, Springer, Berlin, 1999.



\bibitem[La1]{La1}
R.~Lazarsfeld, 
\emph{Lengths of periods and Seshadri constants of abelian varieties},
Math. Res. Lett. \textbf{3} (1996), no. 4, 439--447. 

\bibitem[La2]{La2}
R.~Lazarsfeld, 
\emph{Positivity in algebraic geometry I}, 
Ergebnisse der Mathematik undihrer Grenzgebiete, vol. \textbf{48}. Springer, Berlin (2004).




\bibitem[MP]{MP}
D.~McDuff and L.~Polterovich,
\emph{Symplectic packings and algebraic geometry},
Invent. Math. \textbf{115} (1994), no. 3, 405--434. 

\bibitem[Na1]{Na1}
M.~Nakamaye,
\emph{Seshadri constants on abelian varieties},
Amer. J. Math. \textbf{118} (1996), no. 3, 621--635.

\bibitem[Na2]{Na2}
M.~Nakamaye,
\emph{Seshadri constants and the geometry of surfaces},
J. Reine Angew. Math. \textbf{564} (2003), 205--214.

\bibitem[RT]{RT}
J.~Ross and R.~Thomas,
\emph{A study of the Hilbert-Mumford criterion for the stability
of projective varieties},
J. Algebraic Geom. \textbf{16} (2007), no. 2, 201--255.

\bibitem[St]{St}
A.~Steffens,
\emph{Remarks on Seshadri constants},
Math. Z. \textbf{227} (1998), no. 3, 505--510.




\bibitem[ST]{ST}
T.~Szemberg and H.~Tutaj-Gasińska,
\emph{Seshadri fibrations on algebraic surfaces},
Ann. Acad. Pedagog. Crac. Stud. Math. \textbf{4} (2004), 225--229.




\bibitem[WW]{WW}
K.~Watanabe and M.~Watanabe,
\emph{The classification of Fano $3$-folds with torus embeddings},
Tokyo J. Math. \textbf{5} (1982), no. 1, 37--48. 

\end{thebibliography}
\end{document}